\documentclass[11 pt]{amsart}
\usepackage[applemac]{inputenc}

\usepackage[dvipsnames,svgnames,x11names]{xcolor}
\definecolor{ffmblue}{HTML}{006092}

\usepackage[hyphens]{url}
\usepackage[colorlinks=true,linkcolor=Maroon,citecolor=Maroon,urlcolor=ffmblue,hypertexnames=false,linktocpage]{hyperref}
\usepackage{bookmark}
\usepackage{amsmath,thmtools,mathtools}
\usepackage{dsfont}
\mathtoolsset{showonlyrefs=true}

\newcounter{mparcnt}

\usepackage{fancyhdr}
\usepackage{esint}
\usepackage{enumerate}

\usepackage{pictexwd,dcpic}
\usepackage{graphicx}
\usepackage{caption}
\swapnumbers
\declaretheorem[name=Theorem,numberwithin=section]{thm}

\declaretheorem[name=Lemma,sibling=thm]{lemma}
\declaretheorem[name=Proposition,sibling=thm]{prop}
\declaretheorem[name=Definition,style=definition,sibling=thm]{defn}
\declaretheorem[name=Corollary,sibling=thm]{cor}

\numberwithin{equation}{section}

\newcommand{\ti}{\tilde}

\newcommand{\bs}{\backslash}
\newcommand{\cn}{\colon}
\newcommand{\sub}{\subset}

\newcommand{\bbN}{\mathbb{N}}

\newcommand{\bbR}{\mathbb{R}}

\newcommand{\bbS}{\mathbb{S}}
\newcommand{\bbH}{\mathbb{H}}

\newcommand{\bbE}{\mathbb{E}}

\newcommand{\8}{\infty}

\newcommand{\al}{\alpha}
\newcommand{\be}{\beta}
\newcommand{\ga}{\gamma}

\newcommand{\ka}{\kappa}
\newcommand{\la}{\lambda}

\newcommand{\si}{\sigma}
\newcommand{\Si}{\Sigma}

\newcommand{\vt}{\vartheta}
\newcommand{\Om}{\Omega}
\newcommand{\De}{\Delta}

\newcommand{\Th}{\Theta}

\newcommand{\del}{\partial}
\newcommand{\n}{\nabla}

\newcommand{\fa}{\forall}
\newcommand{\rt}{\sqrt}

\newcommand{\ip}[2]{\left\langle #1,#2 \right\rangle}
\newcommand{\fr}[2]{\frac{#1}{#2}}

\newcommand{\x}{\times}

\DeclareMathOperator{\dive}{div}

\DeclareMathOperator{\dist}{dist}

\DeclareMathOperator{\tr}{tr}

\DeclareMathOperator{\Rc}{Rc}

\DeclareMathOperator{\vol}{vol}
\DeclareMathOperator{\Ar}{Area}

\newcommand{\pf}[1]{\begin{proof}#1 \end{proof}}
\newcommand{\eq}[1]{\begin{equation}\begin{alignedat}{2} #1 \end{alignedat}\end{equation}}

\newcommand{\abs}[1]{\lvert #1\rvert}
\newcommand{\enum}[1]{\begin{enumerate}[(i)] #1 \end{enumerate}}
\newcommand{\enu}[1]{\begin{enumerate}[(a)] #1 \end{enumerate}}

\newcommand{\ra}{\rightarrow}
\newcommand{\hra}{\hookrightarrow}

\newcommand{\mt}{\mapsto}

\newcommand{\hp}{\hphantom}
\newcommand{\q}{\quad}

\begin{document}

\title[Optimisers of reverse isoperimetric problems]{Two properties of optimisers for the reverse isoperimetric problem}
\author{Deniz M. Hamdy and Julian Scheuer}
\address{\flushleft\parbox{\linewidth}{Goethe-Universit\"at\\ Institut f\"ur Mathematik\\ Robert-Mayer-Str.~10\\ 60325 Frankfurt\\ Germany\\{\href{mailto:hamdy@stud.uni-frankfurt.de
}{hamdy@stud.uni-frankfurt.de
}} \\ {\href{mailto:scheuer@math.uni-frankfurt.de}{scheuer@math.uni-frankfurt.de}}}}
\keywords{Reverse isoperimetric problem}
\date{\today}

\begin{abstract}
The reverse isoperimetric problem asks for existence and properties of bounded convex sets in a Riemannian manifold which maxi\-mise the perimeter amongst all those sets of fixed volume which roll freely in a ball of some given radius. If the boundary of the set is of class $C^{2}$, this amounts to a positive lower bound on the principal curvatures and in this class we prove that there are no $C^{2}$-maximisers of the perimeter with prescribed volume. In addition, we prove that a given possibly non-$C^{2}$ maximiser has its smallest principal curvature constant in regions where it is of class $C^{2}$. We prove this result in the Euclidean, spherical and hyperbolic space.
\end{abstract}
\maketitle

%\tableofcontents

\section{Introduction}
Classical isoperimetric problems ask for existence and properties of bound\-ed domains, which minimise the perimeter amongst all sets enclosing a prescribed amount of volume.

A classical result in this regard holds for simple closed curves in the plane, which bound a domain $\Om\sub \bbR^{2}$. There holds 
\begin{equation*}
4\pi\vol(\Omega)\leq \Ar(\del\Om)^{2},
\end{equation*}
where the equality case is reserved for all balls $\Om = B_{R}$.
With the help of the Brunn-Minkowski inequality this result can be generalised to arbitrary dimension: For a bounded set in the $(n+1)$-dimensional Euclidean space with $C^{1}$-boundary (there are further different generalisations to sets with weaker regularity) we have
\begin{equation}
\frac{\Ar(\partial \Omega)}{\Ar(\partial B_{1})} \geq \left(\frac{\vol(\Omega)}{\vol(B_{1})}\right)^{n/(n+1)}.
\end{equation}
 The equality case is again reserved for balls of arbitrary radius.
Also in non-Euclidean spaces there are versions of this result and the literature is vast. We refer the interested reader to the well written survey \cite{Osserman:11/1978} and instead focus on the class of problems that is the main subject of this paper.

The reverse isoperimetric inequality instead asks the question, whether in a certain class of domains with fixed enclosed volume, the perimeter is bounded above and, if this is the case, a (possibly unique) object realises this bound. Certainly, without further constraints, such a bound can not exist. This can easily be seen from the unbounded domain enclosed by the graph of any smooth, positive and integrable function $f\cn \bbR\ra \bbR$ and the $x$-axis. Cutting this domain at $x=\pm R$ and sending $R$ to infinity, we find a sequence of domains in $\bbR^{2}$ with bounded Lebesgue measure but unbounded perimeter. 

The question becomes interesting, once we restrict the admissible class to domains the boundary of which satisfies some bounds on the principal curvatures. For this purpose, let us call a smooth, bounded and convex domain of a Riemannian manifold {\it $\la$-convex}, if all principal curvatures of the boundary are at least $\la$, while we call it {\it $\la$-concave}, if all principal curvatures are at most $\la$. In the non-smooth setting the notions can be defined using an enclosing ball property at every boundary point, see below.
In \cite{ChernovDrachTatarko:09/2019}, Chernov, Drach and Tatarko proved that in the $n$-dimensional Euclidean space with $n\geq 2$, the $\la$-concave so-called {\it sausage body}, i.e. the Minkowski sum of a line segment with a ball of radius $1/\la$, is a maximiser of the perimeter amongst all bodies with this very same volume. 
Amongst all $\la$-convex bodies in three-dimensional Euclidean space, Drach and Tatarko proved in \cite{DrachTatarko:/2023} that the $\la$-lenses, i.e.~the intersection of two spheres of radius $1/\la$, are the maximisers of the perimeter. This also holds in two dimensions, see \cite{BorisenkoDrach:/2014,FodorKurusaVigh:/2016}. 

Hence it appears that the $\la$-convex case is widely open in all dimensions greater than three. Using variational methods, in this paper we prove that we can at least rule out maximisers of class $C^{2}$. 
In addition, our techniques will also imply that at every point of the $C^{2}$-pieces of a maximiser, the smallest principal curvature must be $\lambda$. A related problem involving constant width, the so-called Blaschke-Lebesgue problem, was treated in \cite{AnciauxGuilfoyle:/2011}.
 This also applies to convex sets within the sphere or the hyperbolic space.

Here is our main result, for which we first need some notation. We write $\bbE^{n+1}$, $\bbS^{n+1}$ and $\bbH^{n+1}$ for the Euclidean, spherical and hyperbolic space of sectional curvature $0$, $1$ and $-1$ respectively.
We define $I_{\mathbb{H}^{n+1}}=(1,\infty) $,  $I_{\mathbb{E}^{n+1}}:=(0,\infty)$, $I_{\mathbb{S}^{n+1}}:=(0, \infty)$ and for $\lambda \in I_{\Si}$,
\begin{equation*}
R_{\Sigma}(\lambda) := 
\begin{cases}
1/\lambda & \Si = \mathbb{E}^{n+1} \\
\cot^{-1}(\lambda) & \Si = \mathbb{S}^{n+1} \\
\coth^{-1}(\lambda) & \Si=\mathbb{H}^{n+1}.
\end{cases}
\end{equation*}
This definition is motivated by the observation, that for $\la\in I_{\Si}$, the geodesic sphere with radius $R_{\Si}(\la)$ has constant principal curvatures $\la$. First we recall the definition of $\la$-convexity, compare \cite{DrachTatarko:/2023}.

\begin{defn}[$\lambda$-convexity]\label{defn:convexity}
Let $\la>0$.
\enum{
\item A convex set $K\sub \Si$ is {\it $\la$-convex}, if at any $p\in \del K$, a neighbourhood of $p$ within $\del K$ lies within the mean curvature side\footnote{The one the mean curvature vector points into.} of a complete, totally umbilic hypersurface of curvature $\la$, touching $\del K$ at $p$. 
\item A {\it $\la$-convex body} is a compact, $\la$-convex set with nonempty interior.
\item A smooth hypersurface $M$ of $\Si$ is called {\it strictly $\la$-convex}, if at every point in $M$ the smallest principal curvature of $M$ is larger than $\la$. 
}
In either case we call $K$ resp. $M$ {\it nontrivial}, if it is not a full side of the umbilic hypersurface resp. a piece of it.
\end{defn}

In case that $\la\in I_{\Si}$, this definition says that the set $K$ rolls freely in a ball of radius $R_{\Si}(\la)$ and the latter we will also call {\it supporting ball}, see \autoref{thm:global Blaschke} below.  However, we need to give the more general definition as above, because we want the following main result to encompass $\la>0$ also in the hyperbolic case.

With these definitions at hand, we can formulate our main result.

\begin{thm}\label{thm:main}
Let $n\geq 1$, $\la>0$ and $K$ be a nontrivial compact $\la$-convex body in $\Si=\bbE^{n+1}$, $\Si=\bbS^{n+1}$ or $\Si=\bbH^{n+1}$.
Suppose that the set $K$ has the largest perimeter among all $\la$-convex bodies with the same volume as $K$. Then the following statements are valid:
\enum{
\item $\del K$ is not of class $C^{2}$, i.e. $K$ does not have twice continuously differentiable boundary. 
\item If $\la\in I_{\Si}$, then the smallest principal curvature is $\la$ on every point of the set where $\del K$ is of class $C^{2}$.
}
\end{thm}

The difficulty of the proof lies in the non-openness of the admissible set, which necessarily allows that the smallest principal curvature attains the value $\la$. Hence an ad hoc variation will leave this class and indeed, so far we have not found a variation which fixes volume, increases the perimeter and at the same time respects the admissible class. Finding such a variation is still an interesting open question. Instead, our strategy combines geometrical and variational methods. To prove (i), we will show that every $\la$-convex body, which is not a round sphere, must lie in a non-trivial lens. The latter then can be shown to lie in a sphere of radius less than $R_{\Si}(\la)$, and by a standard argument contains a point, where all curvatures are strictly greater than $\la$. This allows a local variational argument, using the first and second variational formula. To prove (ii), we simply perform the same local variation argument in the supposed smooth region.

\subsection*{Acknowledgment}
This is the version accepted by the Proceedings of the American Mathematical Society.
We would like to thank the anonymous referees for the in depth checking of the manuscript and many valuable suggestions.

\section{Basic notions and auxiliary results}\label{sec:Basics}

\subsection*{Conventions}
In this paper we deal with oriented hypersurfaces $M$ embedded into one of the simply connected spaceforms of constant sectional curvature, which we subsume under the symbol $(\Si,\ip{\cdot}{\cdot})$. 
To fix notation, suppose
\eq{(M,g,\n)\hra (\Si,\ip{\cdot}{\cdot},D)}
to be isometrically immersed with the respective Levi-Civita connections deduced from the metrics. For vector fields $V$ and $W$ on $M$, we use the common abuse of notation that the quantity $D_{V}W$ tacitly assumes local extensions of $V$ and $W$ to the ambient space.

For a given unit normal field $\nu$ along $M$, $h$ denotes its scalar valued second fundamental form,
\eq{
h(V,W)(p)=\ip{D_V W}{\nu},
}
and with $A$ we denote the associated Weingarten map with eigenvalues $\ka_{1}\leq \dots\leq\ka_{n}$. We write $H=\tr(A)$ for the mean curvature.
For brevity, we use Einstein's summation convention and also write
\eq{F_{,i} = \del_{i}F}
for (vector valued) functions and
\eq{T_{;i}:=\nabla_{\partial_{i}}T.}
for tensor fields on $M$.

\subsection*{Variations}
We collect the required facts about variations of hypersurfaces. Most of these results of course do not need $\Si$ to be a spaceform.

\begin{defn}[Variation and normal Variation]
 Let $M \sub\Si^{n+1}$ be a $C^{k}$-hypersurface for some $k\in \bbN\cup\{0,\8\}$ and $J \sub \mathbb{R}$ be an interval containing zero. \enu{
 \item A {\it variation of $M$} is a map $\phi \in C^{k}(M \times J, \Si)$ with the properties
 \enum{
\item $\phi_t:=\phi(\cdot,t)$ is an embedding of a $C^{k}$-hypersurface $M_{t}$ for all $t\in J$,
\item $\phi_0(x) = x $ for all $x\in M$,
}
\item In case $k\geq 1$, a variation is called {\it normal}, if for all $p\in M$,
\eq{\partial_t \phi(p,0) \in N_pM,}
the latter being the normal space of $M$ at $p$. 
 For a given unit normal field $\nu$ on $M$ and its evolutions $\nu(\cdot,t)$ on $M_{t}$ we call the function $v_\phi$, defined by
\eq{
v_{\phi}(p,t)=\ip{{\partial_t \phi}(p,t)}{\nu(p,t)},
}
the \textit{normal velocity} of $\phi$.
\item In case $k\geq 2$,
the function $a_{\phi}$ defined by
\eq{
a_{\phi}(p,t)=\ip{D_t \partial_t\phi(p,t)}{\nu(p,t)},
}
is called the \textit{normal acceleration} of $\phi$. Here $D_{t}$ is the covariant derivative along the curve $\phi(p,\cdot)$.
\item 
We say, a variation $\phi \in C^{0}(M \times J,\Si)$ has \textit{compact support in an open set $\Om\sub M$}, if $\phi$ is a variation of $\Om$ and $\phi_{t}(x) =x$ for all $t\in J$ and all $x$ which lie outside a compact subset of $\Om$.
}
\end{defn}

Here are our required variational formulae. The assumptions are, for comprehensibility, adjusted to our setting and are by no means the most general possible.  

\begin{prop}\label{prop:ev eq}
Let $\phi$ be a variation of the compact hypersurface $M = \del K$, where $K$ is a convex body, i.e. a compact convex set with nonempty interior. Suppose $\Om\sub M$ is open and of class $C^{2}$. Let $\nu$ be the inward unit normal field on $\Om$. Then there hold for all variations $\phi$ of $M$ with compact support in $\Om$, 
\eq{\fr{d}{dt}\vol(K_{t}) = -\int_{M_{t}}v_{\phi},}
\eq{
\frac{d}{dt} \Ar(M_t)=-\int_{M_{t}} v_{\phi}H.
}
In case that $\phi$ is normal, we obtain
\eq{\fr{d^{2}}{dt^{2}}_{|t=0}\vol(K_{t}) = \int_{M_{t}}(v^{2}_{\phi}H - a_{\phi})}
and, if $\Om$ is additionally of class $C^{3}$,
\eq{
\frac{d^2}{dt^2}_{|t=0} \Ar(M_t)=\int_{M_{t}} (-v_{\phi}T(v_{\phi})+H^2v_{\phi}^2-a_{\phi}H),
}
where
\eq{
T:=\Delta +\Rc_{\Si}(\nu,\nu)+\tr(A^2)
}
denotes the stability operator and $\Rc_{\Si}$ is the Ricci tensor of the Riemannian manifold $\Si$.
\end{prop}

\pf{
The first formula can be deduced by differentiating the relation
\eq{\vol(K_{t}) = \int_{0}^{\8}\int_{\bbS^{n}}\mathds{1}_{K_{t}}(r\xi)\vt^{n}(r)\,d\xi dr,}
where $\mathds{1}$ is the indicator function and
\eq{\label{pf:ev eq}\vt(r) = \begin{cases}r, & \Si = \bbE^{n+1}\\
					\sin(r), &\Si = \bbS^{n+1}\\
					\sinh(r), &\Si = \bbH^{n+1}.\end{cases}}
 For the others we display the main ingredients of their proofs for completeness. For brevity write $v = v_{\phi}$ and $a=a_{\phi}$. 
The induced metric satisfies
\eq{\del_{t}\ip{\phi_{,i}}{\phi_{,j}} &= \ip{(\del_{t}\phi)_{;i}}{\phi_{,j}} + \ip{\phi_{,i}}{(\del_{t}\phi)_{;j}}\\
				&=-2vh_{ij}+\ip{(\del_{t}\phi)^{\top}_{;i}}{\phi_{,j}} + \ip{\phi_{,i}}{(\del_{t}\phi)^{\top}_{;j}}
				}
and hence for the area element we get
\eq{\del_{t}\rt{\det g_{ij}} = -vH + \dive_{M_{t}}(\del_{t}\phi)^{\top}.}
This proves the second formula using the divergence theorem on $\Om_{t}$.
The normal velocity evolves according to
\eq{\del_{t}v_{|t=0} = a + \ip{\del_{t}\phi}{D_{t}\nu} = a-\ip{\del_{t}\phi}{\n v}}
and hence we obtain
\eq{\fr{d^{2}}{dt^{2}}_{|t=0}\vol(K_{t}) = -\fr{d}{dt}_{|t=0}\int_{M_{t}}v=\int_{M}(v^{2}H - a),  }
where we used that at $t=0$, $\del_{t}\phi$ is normal.
Finally, we use a computation similar to the proof of \cite[Lemma~2.3.3]{Gerhardt:/2006} and evaluate at $t=0$ to obtain
\eq{\del_{t}H_{|t=0} = \De v + \tr(A^{2})v + \Rc_{\Si}(\nu,\nu)v,}
and compute further,
\eq{\fr{d^{2}}{dt^{2}}_{|t=0}\Ar(M_{t}) &= \int_{M}v^{2}H^{2}-\int_{M}aH - \int_{M}vT(v).
			}
			}

\subsection*{Stability}

Later we have to use the second variation in order to perturb our hypersurfaces, so that the notion of stability naturally comes into play, which we briefly recall here.

\begin{defn}\label{defn:stability}
A hypersurface of constant mean curvature $M\sub \Si$ is called {\it strongly stable}, if for all open and connected sets $\Om\sub M$ and all not identically vanishing $ f\in C^{\8}_{c}(\Om)$ with the property $\int_{M}f = 0$, there holds
\eq{-\int_{\Om}fT(f)>0.}
Here $C^{\8}_{c}(\Om)$ denotes the set of smooth functions on $M$, which have compact support in $\Om$.
\end{defn}

In order to check stability in practice the following lemma is of high value. The idea is taken from the proof of \cite[Prop.~2.5]{ElbertNelli:08/2024}.

\begin{prop}
A hypersurface of constant mean curvature $M\sub \Si$ is strongly stable, if there exists a smooth positive function $u$ with the property $T(u)\leq 0$.
\end{prop}

\begin{proof}
Let $f$ and $\Om$ be as in \autoref{defn:stability} and
define $\rho:=\frac{f}{u}$. Let $\nu$ be a unit normal field on $\Om$. Then we have
\begin{equation*}
\begin{aligned}
-\int_\Omega fT(f) &=-\int_{\Omega} (f\Delta f+\tr(A^2)f^2+\Rc_{\Si}(\nu,\nu)f^2) \\
&= - \int_{\Omega}\rho u(u\Delta \rho  +2g(\n \rho,\n u)+\rho \Delta u+\rho u\tr(A^2)\\
 &\hp{=-\int_{\Om}\rho u((}+\rho u\Rc_{\Si}(\nu,\nu))\\
&=  - \int_{\Omega}(\rho u\left(u\Delta \rho  +2g(\n \rho,\n u)\right)+\rho^2uT(u))\\
&\geq -\int_\Omega (\rho u^2\Delta \rho+2\rho ug(\n \rho,\n u))\\
&=\int_\Omega (g(\n(u^2\rho),\n \rho)-2 \rho ug(\n \rho, \n u))\\
&=\int_\Omega u^2 g(\n \rho,\n \rho)>0,
\end{aligned} 	
\end{equation*}
where for the second to last equality we used $\rho$ to have compact support, and for the strictness of the last inequality, we use that $\rho$ can not be constant on $\Om$, for otherwise $f$ would need a fixed sign, which is impossible due to the integral condition on $f$.
\end{proof}

The sufficient condition from the previous proposition is satisfied in a particular situation, which will be valid in our setting.

\begin{prop}{\cite[Prop~4.4]{ElbertNelli:08/2024}}\label{ElbertNelli}
Suppose a hypersurface $M\sub \Si$ of constant mean curvature with unit normal field $\nu$ satisfies $\ip{\nu}{X}>0$ for some Killing field $X$ of $\Si$, then $M$ is strongly stable.
\end{prop}

Finally, we will make use of a Blaschke-type rolling ball result, a smooth version of which appeared in \cite{Drach:05/2026} and which is the global analogue of the local supporting ball property, see \cite[Lemma~6]{BorisenkoDrach:/2013} and \cite[Thm.~2.2]{DrachTatarko:/2023}.

\begin{thm}\label{thm:global Blaschke}
Let $K \sub \Sigma$ be a $\la$-convex body, where $\la\in I_{\Si}$. Then $K\sub \bar B_{R_{\Si}(\la)}$ for every supporting ball $B_{R_{\Si}(\la)}$.  
\end{thm}

\section{Proof of the theorem}

Now we combine our auxiliary results from \autoref{sec:Basics} to give a proof of the main result through some further lemmata.

\begin{lemma}
Let $K\sub \Si$ be a non-trivial $\la$-convex body, where $\la\in I_{\Si}$. Then $K$ lies in a nontrivial $\la$-lens, i.e. the intersection of two balls $\bar B_{R_{\Si}(\la)}$ with different centres.  
\end{lemma}

\begin{proof}
For $p\in \del K$, denote by $B^{p}$ the supporting ball with radius $R_{\Si}(\la)$ touching $\del K$ at $p$. As $K$ itself is not a ball of radius $R_{\Si}(\la)$ by nontriviality, there must exist $p\in\del K$ and $q\in \del K$, such that $B^{p}\neq B^{q}$ and hence $K$ lies in the nontrivial lens $\bar B^{p}\cap \bar B^{q}$.  
\end{proof}

\begin{lemma}
Let $L\sub \Si$ be a nontrivial $\la$-lens. Then $L\sub B_{\rho}(c)$, for some $c\in \Si$ and $\rho<R_{\Si}(\la)$. Here $B_{\rho}(c)$ is a geodesic ball of radius $\rho$ and center $c$.
\end{lemma}

\pf{
Let $L = \bar B_{R_{\Si}(\la)}(p)\cap \bar B_{R_{\Si}(\la)}(q)$ and define $c$ to be the center on the minimising geodesic $\ga\cn [0,1]\ra \Si$ from $p$ to $q$, i.e. $c = \ga(1/2)$. Fix $z\in L$, use geodesic polar coordinates centred at $z$ and write the metric of $\Si\bs\{z\}$ as $dr^{2} + \vt^{2}(r)\si$, where $\si$ is the round metric on the sphere and $\vt$ is as in \eqref{pf:ev eq}.
Define
\eq{\be(t) = \Th(r(\ga(t))),}
where $\Th'=\vt$. Then there holds
\eq{\ddot{\be} = \Th_{;ij}\dot{\ga}^{i}\dot{\ga}^{j} = \vt'\abs{\dot{\ga}}^{2}.}
For $\Si\in \{\bbE^{n+1},\bbH^{n+1}\}$ this immediately implies that $\be$ is strictly convex and hence attains its maxima on the boundary, i.e.
\eq{\be(t)< \max(\be(0),\be(1))\leq \Th(R_{\Si}(\la)).}
Hence the $\dist_{\Si}(z,c)<R_{\Si}(\la)$, while $z\in L$ was arbitrary.

On the sphere, the positivity of $\vt'$ is not guaranteed, as the lens might not lie in the hemisphere around $z$. We investigate this case separately. There holds
\eq{\ddot{\be} = -\abs{\dot{\ga}}^{2}\be}
and hence for $\tau :=\abs{\dot{\ga}}^{2}t\in[0,\dist_{\bbS^{n+1}}(p,q)]\sub [0,\pi[$,
\eq{\be(t) = a\cos \tau + b\sin\tau }
with suitable coefficients $a$ and $b$. We choose $\Th(r) = -\cos r$ and hence
\eq{\be(0)=\Th(r(p))\leq \Th(R_{\bbS^{n}}(\la))<0\q\mbox{and}\q \be(1)<\Theta(R_{\bbS^{n+1}}(\la))<0.}
We claim that $\be(t)<\Th(R_{\bbS^{n}}(\la))$ for all $t$. Otherwise $\be$ attains an interior maximum at $t_{0}$, at which there must hold 
\eq{-\abs{\dot{\ga}}^{2}\be(t_{0})=\ddot\be(t_{0})< 0,}
where the strict inequality follows from using $\dot\be(t_{0})=0$. Hence the function $\tau \mt a\cos\tau + b\sin\tau$ would have two zeros in $[0,\pi[$, which is not possible. The proof concludes as above.
}

By standard comparison with an enclosing sphere we obtain:

\begin{cor}\label{cor:strictly lambda convex}
Let $K\sub \Si$ be a non-trivial $\la$-convex body of class $C^{2}$, where $\la>0$. Then there exists $p\in \del K$, at which the smallest principal curvature is larger than $\la$. In particular, $\del K$ contains a strictly $\la$-convex hypersurface.
\end{cor}

Note that in the case $\la\leq 1$ and $\Si = \bbH^{n+1}$, we do not need the previous results for \autoref{cor:strictly lambda convex} to hold. Here it follows trivially by construction of a touching compact sphere.

Now we deal with the question about stability. For this purpose we make the quick observation that due to the rich symmetries of the spaceforms, we can (almost) always find a local Killing field with the desired property.

\begin{prop}\label{prop:strongly stable}
Let $M\sub\Si$ be a hypersurface of constant mean curvature, where in the hyperbolic space we assume that $M$ is not contained in a horosphere. Then there exists a point, such that a neighbourhood of $p$ in $M$ is strongly stable. In the Euclidean and spherical case any choice of $p$ works.
\end{prop}

\pf{
Let $\nu$ be a unit normal filed of $M$ and let $p\in M$. In $\bbE^{n+1}$, the translation of $\bbE^{n+1}$ in direction $\nu_{p}$ gives rise to a Killing field $X$ with the property $X_{p} = \nu_{p}$. As $M$ is of class $C^{1}$, the property 
\eq{\ip{\nu}{X}>0}
persists in a neighbourhood of $p$, which is then strongly stable by  
\autoref{ElbertNelli}.
The argument in the sphere is similar, where we use a rotation in direction of the normal. In the hyperbolic space we use the upper half-space model $\bbH^{n+1} = \bbR^{n}\x (0,\8)$ and use the translation in the horizontal direction, which gives $\ip{\nu_{p}}{X}>0$, provided $\nu_{p}$ is not normal to the horizontal planes, which is precisely excluded to hold for all $p$ by demanding not to be tangent to a horosphere.
}

We are now ready to prove \autoref{thm:main}. The argument is of a quite classical variational nature as for example in \cite{BarbosaCarmo:/1984}. We prove that in a strictly $\la$-convex region we can increase the perimeter in a volume preserving fashion. This then concludes both statements of the theorem, after employing our supplementary arguments from before. 

\begin{prop}\label{prop:main}
Let $K\sub\Si$ be a $\la$-convex body with $\la\in I_{\Si}$, so that $\del K$ contains a strictly $\la$-convex $C^{2}$-hypersurface. Then there exists a $\la$-convex body $\hat K$ with the properties
\eq{\vol(K) = \vol(\hat K)\q\mbox{and}\q \Ar(\del \hat K)>\Ar(\del K).} 
\end{prop}

\begin{proof}

We will construct normal variations of the form
\eq{\phi\cn \del K\x J&\ra \Si}
with compact support in the strictly $\la$-convex  $C^{2}$-hypersurface $\Om\sub \del K$, where without loss of generality $\Om$ is so small, that there exists $p_{0}\in \Om$ with the property
\eq{\ip{\hat\nu}{\nu_{p}}>0\q \fa p\in\Om,}
where $\nu_{p}$ is the inward normal at $p$ and $\hat\nu$ is a $C^{2}$-extension of $\nu_{p_{0}}\in N_{p_{0}}\Om$ to a vector field $\hat\nu$ along $\Om$. 
 We distinguish two cases.

{\bf Case 1:} $H$ is not constant in $\Om$. Then pick open sets $\Om_1\sub\Om$ and $\Om_{2}\sub \Om$ with the property
\eq{
H_1:=\sup_{\Om_1}H<\inf_{\Om_2}H=:H_2.
}
Let $F\in C_{c}^{2}(\Om_{1})$ and $G\in C^{2}_{c}(\Om_{2})$ be positive with  
\eq{\int_\Omega F\ip{\hat\nu}{\nu} =\int_\Omega G\ip{\hat\nu}{\nu} =1.} 
We start with an auxiliary family of embeddings of $\del K$, 
\eq{\psi\cn \del K\x J_{1}\x J_{2}\ra \Si }
of the form
\eq{\psi(p,t,s) = \exp_{p}(tF(p)\hat\nu_{p} + sG(p)\hat\nu_{p}),}
where $\hat\nu$ is specified above.
For small $t$ and $s$, $\psi(\cdot,t,s)$ is an embedding of $\del K$ onto the boundary $\del K_{t,s}$ of a $\la$-convex body $K_{t,s}$, as can be checked locally from \autoref{defn:convexity} using a case distinction between points inside $\Om$ and outside $\Om$: Indeed, on points outside $\Om$ the variation does not change $\del K$ at all, so the local touching sphere condition from \autoref{defn:convexity}(i) is still valid outside $\Om$. On points $p\in\Om$ we are in the strictly $\la$-convex region, so for small $s$ and $t$ the image $\Om$ is strictly $\la$-convex by continuity, so the touching sphere condition is valid, as can be seen from a local Taylor expansion around $p$.

We compute
\eq{\del_{s|(t,s)=0}\vol(K_{t,s}) = -\int_{\Om}G\ip{\hat\nu}{\nu}<0,}
and hence by the implicit function theorem we are given a smooth function $b=b(t)$, such that
\eq{\label{pf:main 1}\vol(K_{t,b(t)}) = \vol(K).}
Hence the variation
\eq{\phi(p,t):= \psi(p,t,b(t))}
is volume preserving and we obtain
\eq{\del_{t}\phi(p,t) = (F(p) + b'(t)G(p))D\exp_{p}(tF(p)\hat\nu_{p}+b(t)G(p)\hat\nu_{p})\hat\nu_{p}.}
From \autoref{prop:ev eq} and the Gauss lemma for the exponential map we deduce
\eq{0=\del_{t}\vol(K_{t,b(t)}) = -\int_{\Om}F\ip{\hat\nu}{\ti N(\cdot,t)}_{p} - b'(t)\int_{\Om}G\ip{\hat\nu}{\ti N(\cdot,t)}_{p}, }
where $\ti N(p,t) = D\exp_{p}(tF(p)\hat\nu_{p} + b(t)G(p)\hat\nu_{p})^{-1}N(p,t)$ and $N(\cdot,t)$ is the inward unit normal of $\Om_{t,b(t)}$.
This implies
\eq{b'(t) = -\fr{\int_{\Om}F\ip{\hat\nu}{\ti N(\cdot,t)}_{p}}{\int_{\Om}G\ip{\hat\nu}{\ti N(\cdot,t)}_{p}}. }
 From \autoref{prop:ev eq} we obtain
\eq{\fr{d}{dt}\Ar(\del K_{t,b(t)}) = -\int_{\del K_{t,b(t)}}(F(p)+b'(t)G(p))\ip{\hat\nu}{\ti N(p,t)}_{p}H}
and hence, from the fact $b'(0)=-1$ we deduce
\eq{
\frac{d}{dt}_{|t=0}\Ar(\del K_{t,b(t)}) &= \int_\Omega H(G-F)\ip{\hat\nu}{\nu_{p}}\\
		&=\int_{\Om_2} GH\ip{\hat\nu}{\nu} - \int_{\Om_1}FH\ip{\hat\nu}{\nu}\\
			& > H_2 \int_{\Om_2} G\ip{\hat\nu}{\nu} - H_{1}\int_{\Om_1}F\ip{\hat\nu}{\nu}\\
			& >0,
}
so that we can increase the perimeter along this volume preserving deformation.

{\bf Case 2:} $H$ is constant in $\Om$. Then this first variation of perimeter vanishes at $t=0$, and we have to consider the second variation. Therefore it is desirable to work with normal variations, and hence we slightly modify the proof of Case 1, by simply using $\hat\nu=\nu$. Note that this does not cause a drop of regularity of $\psi$, because due to $H$ being constant, $\Om$ is of class $C^{\8}$ by standard elliptic regularity theory. To quickly see this, note that the constant mean curvature equation of a hypersurface in a space form always is of the form
\eq{a^{ij}(x,u,Du)D^{2}_{ij}u + b(x,u,Du) = H,}
see \cite[Equ.~(1.5.10)]{Gerhardt:/2006}, where $u$ is the local graph function representing the hypersurface. Since $u$ is assumed to be $C^{2}$, the functions $a^{ij}(x,u,Du)$ and $b(x,u,Du)$ are of class $C^{0,\al}$. Applying \cite[Thm.~6.17]{GilbargTrudinger:/2001}, we infer $u\in C^{2,\al}$ and bootstrapping then gives $u\in C^{\8}$. Note that the $\Om$ in that reference is not our $\Om$, but a Euclidean domain, over which our $\Om$ is represented by the graph function $u$.
 
 Now, redoing the above proof with $\hat\nu=\nu$, we see that from the constancy of $H$ we obtain that the first variation of perimeter vanishes. Write
\eq{v(p,t) = (F(p)+b'(t)G(p))\ip{\nu}{\ti N}_{p}.} From \autoref{prop:ev eq} we obtain
\eq{0 = \fr{d^{2}}{dt^{2}}_{|t=0}\vol(K_{t,b(t)}) = \int_{\Om}(v^{2}H - a), }
where $a(p,t) = \ip{D_{t}\del_{t}\phi(p,t)}{N(p,t)}$, which leads to
\eq{\fr{d^{2}}{dt^{2}}_{|t=0}\Ar(\del K_{t,b(t)}) = \int_{\Om}(-v T(v) + H^{2}v^{2} - aH) = -\int_{\Om}v T(v)>0,  }
if we choose $F$ with $\int_{\Om}F=0$ and $F$ not identically zero, which guarantees $b'(0) = 0$ and $v(p,0) = F(p)$. Hence we can increase the perimeter for small $t$.

\end{proof}

We can finalise the proof of the main result as follows:

\begin{proof}[Completion of the proof of \autoref{thm:main}]
(i) If $\del K$ is of class $C^{2}$, $\la$-convex and not a sphere of radius $R_{\Si}(\la)$, then our previous arguments show that $\del K$ contains a strictly $\la$-convex open $C^{2}$-hypersurface with $\la\in I_{\Si}$. \autoref{prop:main} then proves that $K$ can not have the largest perimeter amongst all $\la$-convex sets of the same volume.

(ii) If the smallest curvature was not equal $\la$ somewhere on the $C^{2}$-portion, it would have to be greater than $\la$ there, due to the $\la$-convexity of $K$. As in this case we assume $\la\in I_{\Si}$, we may employ \autoref{prop:main} to finalise the argument.
\end{proof}

\bibliographystyle{amsplain}
\providecommand{\bysame}{\leavevmode\hbox to3em{\hrulefill}\thinspace}
\providecommand{\MR}{\relax\ifhmode\unskip\space\fi MR }
\providecommand{\MRhref}[2]{%
  \href{http://www.ams.org/mathscinet-getitem?mr=#1}{#2}
}
\providecommand{\href}[2]{#2}

\end{document}